\documentclass[12pt]{amsart}
\usepackage{graphicx,tikz,tikz-cd}
\usepackage{amssymb,amscd, MnSymbol}
\usepackage{setspace,fullpage}
\usepackage{enumerate}
\usetikzlibrary{arrows, automata,chains,fit,shapes}

\newcommand{\bea}{\begin{eqnarray*}}
\newcommand{\eea}{\end{eqnarray*}}

\newcommand{\bm}{\begin{pmatrix}}
\newcommand{\fm}{\end{pmatrix}}

\DeclareMathOperator{\diam}{diam}

\newcommand\D{\mathcal D}
\newcommand\Z{\mathbb Z}
\newcommand\N{\mathbb N}
\newcommand\Q{\mathbb Q}
\newcommand\R{\mathbb R}

\newtheorem{theorem}{Theorem}[section]

\newtheorem{lemma}[theorem]{Lemma}
\newtheorem{corollary}[theorem]{Corollary}

\title{Non-rectifiable Delone sets in $SOL$ and other solvable groups}

\author{Tullia Dymarz}
\author{Andr\'es Navas}

\begin{document}

\maketitle

Given a lattice $\Gamma \subset SOL$, we  show that there is a coarsely dense subset 
$\mathcal{D} \subset \Gamma$ that is not 
{biLipschitz} equivalent to $\Gamma$. 
We also prove similar results 
for lattices in certain higher rank abelian-by-abelian groups
and for the solvable Baumslag-Solitar groups.

\section*{Introduction} 

A Delone set in a metric space is a uniformly discrete, coarsely dense subset. This means that there 
exist positive constants $C,D$ such that any two distinct points of this set are at distance $\geq D$, 
and any point in the space is at distance $\leq C$ of some point in this subset. Typical examples 
of Delone sets are co-compact lattices in Lie groups. In this work, we are interested in the biLipschitz 
equivalence of Delone subsets of  certain solvable groups. 

Let us first recall that there is a big difference between the non-amenable and amenable cases. In the 
former one, a general result of Whyte \cite{W} establishes biLipschitz equivalence of any two Delone 
subsets of the same non-amenable finitely-generated group. (Actually, this holds in any non-amenable 
space provided  some uniform condition on the geometry is satisfied.) For the amenable case, the situation 
is unclear. An easy argument shows that any two Delone subsets of $\mathbb{R}$ are biLipschitz 
equivalent. However, this is false in $\R^d$ for $d \geq 2$, as was shown by Burago-Kleiner \cite{BK} 
and, independently, by McMullen \cite{M}. Although not explicitly stated in these works, it turns out 
that their examples may be realized as Delone subsets of $\Z^d$ 
(see \cite{CN}). This motivates the next general question:

\vspace{0.1cm}

\noindent{\bf Question.} Let $\Gamma$ be a finitely-generated infinite amenable group that is not a finite 
extension of $\mathbb{Z}$. Does $\Gamma$ contain a Delone subset that is not biLipschitz 
equivalent to $\Gamma$ itself?

\vspace{0.1cm}
We call such a Delone set \emph{non-rectifiable}.
The first non-rectifiable Delone sets in amenable groups other that $\Z^d$ were constructed by the first author in 
\cite{D} by showing that any lamplighter group $\Gamma= F \wr \Z$ with $|F| < \infty$ contains  finite-index subgroups
that are not biLipschitz equivalent to $\Gamma$ itself. These examples were further generalized in \cite{DPT} to the 
so-called Diestel-Leader groups. 
All lattices in the three dimensional solvable Lie group $SOL$ are biLipschitz equivalent so, in particular, for a given lattice $\Gamma$ it cannot have non-rectifiable subgroups. In this work we show that, nevertheless, there do exist non-rectifiable Delone subsets in any lattice $\Gamma \subset SOL$. Additionally, we prove the same results for lattices in certain abelian-by-abelian solvable Lie groups, as well as for the solvable Baumslag-Solitar groups.


To construct the required subsets, we combine two ingredients. First, we use slight modifications of 
the combinatorial arguments of \cite{CN}, which correspond to discretized versions of the arguments in \cite{BK,M} 
and perfectly fit in our context. Second, we crucially use the description of all quasi-isometies of the 
underlying group, which is available in each case. It is this last issue that unables us to treat the 
general case of the question above, so that the complete answer remains unclear.

\subsection*{Outline} The three different families of groups that we consider are each treated in separate sections; namely, Section \ref{sol:sec} for lattices in $SOL$, Section \ref{hr:sec} for lattices in certain higher-rank abelian-by-abelian groups, and Section \ref{bs:sec} for solvable Baumslag-Solitar groups. The three theorems we prove are the following:

{\bf Theorem \ref{sol:thm}} \emph{Any lattice in $SOL$ contains a non-rectifiable Delone subset. }

{\bf Theorem \ref{hr:thm}} \emph{Any lattice in a boundary one-dimensional, even-scaling, abelian-by-abelian Lie group contains non-rectifiable Delone subsets.}

(See Section \ref{hr:sec} for a definition of \emph{boundary one-dimensional, even-scaling,} abelian-by-abelian Lie groups.)

{\bf Theorem \ref{bs:thm}} \emph{For all m, the solvable Baumslag-Solitar group $$BS(1,m)=\left< t, a \mid tat^{-1} =a^m\right>$$ contains non-rectifiable Delone subsets.}

Section \ref{sol:sec}, the case of lattices in $SOL$, serves as an introduction to the more general case of lattices in abelian-by-abelian Lie groups. (Indeed, $SOL$ itself is a boundary one-dimensional, even-scaling, abelian-by-abelian Lie group).  For any Delone set $\D$ and lattice $\Gamma$, a biLipschitz equivalence $\D \to \Gamma$ gives rise to a quasi-isometry of the ambient Lie group. Quasi-isometries of these Lie groups are understood by the work of Eskin-Fisher-Whyte in \cite{EFW1,EFW2} and Peng in \cite{P1,P2}. Their results imply that certain \emph{box} F\o lner sets are almost preserved up to a quasi-isometry. This allows us to construct $\D$ by removing points from $\Gamma$ in such a way that any quasi-isometry induced by a biLipschitz map $\D \to \Gamma$ would violate the conditions of the quasi-isometric rigidity theorems. 
Solvable Baumslag-Solitar groups are not lattices in any real Lie groups but their quasi-isometries have a similar structure to those of $SOL$. Again, we show that the images of certain special F\o lner sets are almost preserved under quasi-isometries, and this allows us to apply a similar proof to the one for lattices in $SOL$.

\subsection*{Acknowledgements} 
This paper was initiated when the authors were attending the quarter
program �Random walks on groups� at the Institute of Henri Poincar\'e from January to
March 2014. They would like to thank IHP for its hospitality and wonderful working
conditions. The first author acknowledges support from National Science Foundation grant
DMS-1207296, and would also like to thank David Fisher for useful conversations.  
The second author acknowledges support from CONICYT's Anillo Research Project 1103 DySyRF, 
and from CNRS (UMR 8628, Univ. d’Orsay) via the ERC starting grant 257110 RaWG.

\section{Delone sets in $SOL$}\label{sol:sec}


The solvable Lie group $SOL=\R^2 \rtimes \R$ with the coordinates $(x,y,t)$ has multiplication rule 
$$ (x_1,y_1,t_1)* (x_2,y_2,t_2) := (x_1 + e^{t_1}x_2,\ y_1 + e^{-t_1}y_2,\  t_1+t_2)$$
and left-invariant metric 
$$ds^2=e^{-2t}dx^2 + e^{2t}dy^2 + dt^2,$$ 
which gives rise to a distance function quasi-isometric to 
$$d \big( (x_1,y_1,t_1), (x_2,y_2,t_2) \big) 
:= e^{-\frac{t_1+t_2}{2}} |x_1 - x_2| 
+ e^{\frac{t_1+t_2}{2}} |y_1 - y_2| + |t_1 - t_2|.$$
A typical example of a lattice in $SOL$ is given by $\Gamma=\Z^2 \rtimes \Z$, where the action of $\Z$ 
is given by the matrix $ \left(\begin{smallmatrix} 2 & 1 \\ 1 & 1\end{smallmatrix}\right)$. 
(Indeed, any diagonalizable matrix with eigenvalues {of norm $\neq 1$} will also give a lattice in $SOL$.)

In the next section we explain the construction of the Delone set $\D$ and then in the following sections we prove that it is indeed non-rectifiable. 

\subsection{Tiling}\label{tiling:sec} In this section we construct the non-rectifiable Delone set $\D$. 
Let $$S_N := [0,N)\times[0,1)\times [0,\log{N})$$ be a subset of $SOL$. 
We can tile the larger set
$S_{N^2}$ by $2N$ translates of $S_N$ by letting 
$${S}_{N^2}=  \bigsqcup_{i,k} g_{k,i} {S_N}, \quad i=1,2\textrm{ and }  k=1,\ldots, N,$$ 
where $g_{k,1} := ((k-1)N, 0,0)$ and $g_{k,2} := (0, \frac{(k-1)}{N}, \log{N})$. 
In other words, $S_{N^2}$ is the disjoint union of 
$$T_{k,1}:=g_{k,1}S_N=[(k-1)N,kN) \times [0,1) \times [0, \log{N})$$
and
$$T_{k,2}:=g_{k,2}S_N=[0,N^2) \times [\frac{k-1}{N},\frac{k}{N})  \times [ \log{N}, 2 \log{N}).$$ 
Now, inductively, for any $M := N^{2^m}$ we can tile $S_M$ by $2^m N^{2^m-1}$ copies of $S_N$. 

Fix $\Gamma$ a lattice in $SOL$. We will construct our Delone set $\D$ inductively by fixing $N_0$ and picking four subsets $Q_j^0\subset S_{N_0}\cap \Gamma$ for $j=1,2,3,4$
and translating them by $g_{k,i}$ to define four new sets $Q_j^1 \subset S_{N_0^2} \cap \Gamma $. We have to be careful here since $g_{k,i} (S_N \cap \Gamma)$ is not 
necessarily a subset of $\Gamma$. However, since $\Gamma$ is coarsely dense in $SOL$, there is always a 
$\gamma_{k,i} \in \Gamma$ with 
$$d_{SOL}(\gamma_{k,i}, g_{k,i}) < C,$$
where $C$ depends only on the coarse density of $\Gamma$ in $SOL$. 
Therefore, we can first construct a Delone set $\hat{\mathcal{D}}$ that is at most distance $C$ from a subset of $\Gamma$.  
Then by moving points in $\hat{\mathcal{D}}$ at most distance $C$ we get the desired Delone set $\D \subset \Gamma$.  We will make sure that the initial subsets $Q_j^0$ are chosen to be at least distance $C$ from the boundary of $S_{N_0}$ 
{so that the natural bijections $\gamma_{k,i} x \mapsto g_{k,i} x$ between $\gamma_{k,i}Q_j^0$ and $g_{k,i}Q_j^0$ do not overlap.}


{\bf Volumes.} In order to make volume computations more feasible, we fix $N_0$ an even integer and pick $n$ points from $\Gamma \cap S_{N_0}$. 
We denote this set $Q^0_0$ (i.e. $|Q_0^0|=n$). The $Q_j^0$ for $j=1,2,3,4$ will all be subsets of $Q_0^0$.
The number $n$ will be chosen later and it will determine how large $N_0$ must be. 
As mentioned above, we pick these points to be at least distance $C$ from the boundary of $S_{N_0}$. This ensures that for $g \in SOL$ 
and $\gamma \in \Gamma $ with $d_{SOL} ( g, \gamma)< C$, we  have
$$|Q^0_0|=|gQ^0_0|=|\gamma Q^0_0 \cap g S_{N_0}|=|\gamma Q^0_0|. $$

If we set $Q^1_0=\bigsqcup_{i,k} g_{k,i} {Q^0_0}$ and then inductively define 
$Q^m_0$, we find that the volume of {$Q^m_0$} is given by
$$|Q^m_0 |=2^m N_0^{2^m-1} |Q^0_0|=M \log_{N_0}{(M)} \frac{n}{N_0},$$
where $M=N_0^{2^m}$.

{\bf Basic tiles.} 
We fix $d_1,d_2,d_3,d_4$ rational constants in $(0,1]$ that satisfy $d_1 < d_2$ as well as  
$$d_3 =\frac{3d_1-d_2}{2} \quad \textrm{and} \quad d_4 = \frac{3d_2 -d_1}{2}.$$
(For instance, pick $d_1=1/3, d_2=1/2,d_3=1/4, d_4=7/12$.) Notice that 
$d_3 < d_1 < d_2 < d_4$. We choose 
$n$ so that $d_j n \in \N$ for each $j = 1,\ldots,4$.
We define basic tiles $Q_j^0 \subset Q^0_0 \subset S_{N_0}$, $j = 1, \ldots, 4$, with 
$$|Q^0_j|=d_j |Q^0_0|=d_j n,$$
by arbitrarily picking the required number of points from $Q^0_0$.


{\bf Inductively defined tiles.} Define $Q_j^1$ for $j=1,2,3,4$ as follows: Consider 
$$T_{k,1} := [(k-1){N_0},k{N_0}) \times [0,1) \times [0, \log{{N_0}})$$
 for $k=1,\ldots , {N_0}$, as before. If $k$ is odd, set $Q_1^1 \vert_{T_{k,1}} = Q_2^1 \vert_{T_{k,1}} := g_{k,1} Q^0_1$, and if $k$ is even, 
 set $Q_1^1 \vert_{T_{k,1}} = Q_2^1 \vert_{T_{k,1}} := g_{k,1} Q^0_2$. (Here, as before, $g_{k,1} \in SOL$ is the element that takes 
 $S_{N_0}$ to $T_{k,1}$.) 
 
Next, consider 
$$T_{k,2} := [0,{N_0}^2) \times [\frac{k-1}{{N_0}},\frac{k}{{N_0}})  \times [ \log{{N_0}}, 2 \log{{N_0}})$$
for $k=1,\ldots , {N_0}$. Then set $Q_1^1\vert_{T_{k,2}} := g_{k,2} Q^0_3$ 
and $Q_2^1\vert_{T_{k,2}} := g_{k,2} Q^0_4$.


Now we compute the volume of these sets:
\bea
|Q_1^1|&=& \frac{{N_0}}{2}|Q_1^0| +  \frac{{N_0}}{2}|Q_2^0| + {N_0} |Q_3^0| \\
&=& \frac{d_1 }{2}{N_0}|Q^0_0| +  \frac{d_2 }{2}{N_0}|Q^0_0| + d_3{N_0} |Q^0_0| \\
&=& \left(\frac{d_1 }{2} +  \frac{d_2 }{2} + d_3\right) \frac{1}{2} |Q^1_0| \\
&=& d_1 | Q^1_0|, 
\eea
and similarly, 
\bea
|Q_2^1|&=& \frac{ {N_0}}{2}|Q_1^0| +  \frac{{N_0}}{2}|Q_2^0| + {N_0} {|Q_4^0|}\\
&=& \left(\frac{d_1 }{2} +  \frac{d_2 }{2} + d_4\right) \frac{1}{2} |Q^1_0|\\
&=& d_2 | Q^1_0|,
\eea
where the last equality in each case comes, respectively, from 
$$\left(\frac{d_1 }{2} +  \frac{d_2 }{2} + d_3\right) \frac{1}{2}=d_1 \quad 
\textrm{and} \quad \left(\frac{d_1 }{2} +  \frac{d_2 }{2} + {d_4}\right) \frac{1}{2}=d_2.$$

To define $Q_3^1$ and $Q_4^1$ simply set $Q_3^1:=\bigsqcup_{i,k} g_{k,i} {Q_3^0}$ 
and $Q_4^1 := \bigsqcup_{i,k} g_{k,i} {Q_4^0}$, 
so that trivially $|Q^1_3|=d_3 |Q^1_0|$ and $|Q^1_4|=d_4 |Q^1_0|$.

Now, we repeat this process to inductively define $Q_j^m$ for $j=1, \ldots, 4$, 
with 
$$|Q_j^m|=d_j |Q_0^{m}|.$$

Next, we let $\hat{\D}$ be the union of all $Q^1_j$. 
This defines $\hat\D$ on $\Gamma \cap U$, where 
$$U :=  \{(x,y,t) \mid x,t \mbox{ non negative, } y \in [0,1] \}.$$ 
We then extend $\hat\D$ to 
$SOL \setminus U$ simply by taking $\hat\D=\Gamma$ therein. Finally, we let $\D$ be the subset of 
$\Gamma$ obtained by translating points of $\hat\D$ a distance $\leq C$, as previously explained. 

\subsection{Quasi-isometries.}

Any $K$-biLipschitz map $f:\mathcal{D} \to \Gamma$ extends to a $(K,C)$ quasi-isometry of $SOL$, 
where $C$ depends on the coarse density constants of $\mathcal{D}$ and $\Gamma$. By Eskin-Fisher-Whyte 
\cite{EFW1, EFW2}, this quasi-isometry is, up to an isometry that permutes the first two factors, at bounded distance $A=A(K,C)$ from a 
``companion'' quasi-isometry $F$ of the form
$$ F(x,y,t)=( f_1(x), f_2(y), t),$$
where $f_1, f_2$ are $L$-biLipschitz maps of $\R$. As long as $f$ coarsely fixes the identity, 
the constant $L$ depends only on $K$ and $C$.  
Here we use the term ``coarsely fixes" the identity to mean that it maps 
the identity at most distance $C$ from the identity in the image. We can always arrange for $f$ to coarsely fix the identity by composing 
$f$ with left multiplication by a group element in $\Gamma$. We will also assume that $f$ does not permute the first two 
factors since this does not change the arguments in the proof.
Henceforth, companion quasi-isometries 
will be always chosen with these properties.   


\subsection{F\o lner sets.} Recall that a \emph{F\o lner sequence} in a finitely generated group $\Gamma$ is a sequence of finite sets $S_i$ such that for all $R>0$ 
$$\lim_{i \to \infty} \frac{ |\partial_R S_i|}{|S_i|} = 0$$
where 
$$\partial_R S_i := \{ x \in \Gamma \mid d_\Gamma(x ,\Gamma\setminus S_i) \leq R \textrm{ and } d_\Gamma(x , S_i)\leq R \}$$
and $d_\Gamma$ is any metric quasi-isometric to a word metric on $\Gamma$.
In a Lie group, the definition is similar but with finite sets replaced with compact sets and counting measure by volume (Haar measure). 
Although the next terminology is not very precise, any set that belongs to a F\o lner sequence will be called a \emph{F\o lner set}. 

In the next lemma, we show that in a Lie group $G$, any F\o lner set $S \subset G$ defines a F\o lner set $\bar{S}:=S \cap \Gamma$ in $\Gamma$ for any uniform lattice $\Gamma \subset G$.
To make computations easier we chose $d_\Gamma$ to be the restriction of the metric $d_G$ on $G$.
This allows us to compare $|S|$ with $|S\cap \Gamma|$ and $ |\partial_R S|$ with $|\partial_R (S\cap \Gamma)|$. We also assume that $|G/\Gamma|=1$. 
%

\begin{lemma}\label{compare:lemma} Let $D := \diam(G/\Gamma)$. 
For any set $S$ and any $R\gg D$, we have that 
$$ \frac{|\partial_{R-2D} S|}{ |S| + |\partial _D S|}\leq \frac{|\partial_R (S\cap \Gamma)|}{|S\cap \Gamma|} 
\leq \frac{|\partial_{R+D} S|}{ |S| - |\partial _D S|}.$$
\end{lemma}
\begin{proof} 
First we claim that for any $S$,
\begin{equation}\label{needed-later}
|S\setminus \partial_{D} S| \leq |S\cap \Gamma| \leq |S\cup \partial_D S|.
\end{equation}
To see this, let $E$ be a fundamental domain, containing the identity, for the action of $\Gamma$ 
on $G$ (so that $|E|=1$ and $\diam(E)=D$), and let $S'=(\Gamma\cap S)\cdot E$. Then 
$|S'|= |S \cap \Gamma|$ and 
$$S \setminus \partial_{D} S \subset S' \subset \partial_D S \cup S.$$  
The second inclusion is clear since $\diam(E)=D$ and so for any $\gamma \in S\cap \Gamma$ we must have $\gamma \cdot E \subset \partial_D S \cup S$.
The first inclusion follows from the fact that any $x \in S\setminus S'$ must lie in some $\gamma \cdot E$, in which case $\gamma \in \partial_D S\setminus S$, and since $d(\gamma, x)\leq D$, we must have $x \in \partial_D S$. 

This claim gives us
$$ |S| - |\partial _D S| \leq |S \cap \Gamma|  \leq |S| + |\partial_D S|.$$
Futhermore, for any $R\gg D$ we have that
$$|\partial_{R-2D} S | \leq |\partial_{R-D} S \cap \Gamma | \leq | \partial_R (S \cap \Gamma) |\leq |\partial_R S \cap \Gamma| \leq |\partial_{R+D} S|,$$
and therefore
$$ \frac{ |\partial_{R-2D} S|}{ |S| + |\partial _D S| }\leq \frac{|\partial_R(S\cap \Gamma)|}{|S\cap \Gamma|}\leq \frac{ |\partial_{R+D} S|}{ |S|-|\partial _D S|},$$
as desired.
\end{proof}

We now define a family of preferred F\o lner sets in $SOL$  (and hence $\Gamma$). 
A \emph{standard} F\o lner set in $SOL$ is given by
$$U_r \times U_s \times [-\log s, \log r)$$
where $U_r$ and $U_s$ are intervals of length $r$ and $s$ respectively. 
Note that $S_{N}$ defined in the previous section is a preferred F\o lner set with $r={N}$ and $s=1$.
As long as $rs > 1$, the volume of this set is  $rs \log (rs)$, and the volume of the $R$-boundary 
is bounded by $\bar{C}rs$, where $\bar{C}$ depends on $R$. For the case of $S_N$, this is 
proved in detail just after Corollary \ref{sol:corollary}  in Section~\ref{hr:sec}.


We also need to estimate the boundary volume ratio for images of standard F\o lner sets; namely sets of the form 
$$U_{ar} \times U_{bs} \times [-\log s, \log r),$$
with $a,b>0$. 
%
%
Any set of this form 
is also a F\o lner set as long as we fix 
$a,b$ and let $rs \to \infty$. 
Again, this follows from Corollary \ref{sol:corollary} in the next section.
The following lemma will be applied with $G=SOL$ in this section and $G$ a higher-rank abelian-by-abelian solvable Lie group in the next section.

\begin{lemma}\label{first:lemma} For any  F\o lner set $S\subset G$, suppose that $f:\mathcal{D} \to \Gamma$  and $F: G \to G$ are bounded distance $A$ apart on $\mathcal{D}$. Then the following inequality holds: 
$$ |F(S)| - |\partial_{A+D} F(S)|  \leq |f(\mathcal{D} \cap S)| \leq |F(S)| + |\partial_{A+D} F(S)|.$$
\end{lemma}

\begin{proof}
Since $f$ is bounded distance $A$ from $F$, the most that we can gain or lose by replacing $F$ with 
$f$ comes from the $A$-boundary of the F\o lner set $F(S)$. Specifically,  
$$\Gamma \cap (F(S) \setminus \partial_A(F(S)))\subset f(\D \cap S) \subset (F(S) \cup \partial_A(F(S))) \cap \Gamma.$$
By {(\ref{needed-later})}, 
since $F(S) \cup \partial_A(F(S))$ is also a F\o lner set, this yields   
\bea |(F(S) \cup \partial_A(F(S))) \cap \Gamma| &\leq& |F(S) \cup \partial_A(F(S))| + |\partial_D ( F(S) \cup \partial_{A+D}(F(S))|  \\
&\leq& |F(S)|+ |\partial_{A+D}(F(S))|
\eea
and 
\bea
|(F(S) \setminus \partial_A(F(S))) \cap \Gamma| &\geq& |F(S) \setminus \partial_A(F(S))| - |\partial_D ( F(S) \setminus \partial_A(F(S))|  \\
&\geq& |F(S)|- |\partial_{A+D}(F(S))|,
\eea
which proves the lemma. 
\end{proof}

\subsubsection{Key lemma} 
Roughly speaking, the following lemma states that if the intersection of 
$\D$ with two standard F\o lner sets in $SOL$ of the same size is radically different,  
then the companion quasi-isometry to any biLipschitz map $f: \D \to \Gamma$ 
must map these F\o lner sets to boxes of sufficiently different sizes. This is 
the lemma that motivates the construction of $\D$ in the previous section. 

\begin{lemma}\label{getM0}
There exist $M_0 > 0$ and $\epsilon_0 > 0$, both depending on $K,C$, satisfying the following: 
Let $M := N_0^{2^m}> M_0$ and suppose $S^1,S^2$ are translates of the standard F\o lner set $S_M$ with
$|\mathcal{D}\cap S^1|=d_1|Q^m_0|$  and $ |\mathcal{D}\cap S^2|=d_2|Q^m_0|$.  
Suppose $f: \mathcal{D} \to \Gamma$ is $K$-biLipschitz, and that the companion quasi-isometry $F$ satisfies 
{\color{black}
$$F(S^1) \mbox{ is isometric to } U_{rM}\times U_s \times [0, \log{M})$$ 
and 
$$F(S^2) \mbox{ is isometric to } U_{r'M}\times U_s \times [0, \log{M}).$$}Then $|r-r'| > \epsilon_0$. 
\end{lemma}



\begin{proof} 
By Lemma \ref{first:lemma} we have that
 $$ |F(S^i)| - |\partial_{A+D} F(S^i)| \leq |f(\mathcal{D} \cap S^i)| \leq |F(S^i)| + |\partial_{A+D} F(S^i)|$$
for $i=1,2$.
We rewrite this as 
 $$ |F(S^i)| - |\partial_{A+D} F(S^i)| \leq d_i |Q^m_0| \leq |F(S^i)| + |\partial_{A+D} F(S^i)|,$$
so that 
\begin{equation}\label{eqneqn} \frac{|F(S^i)|}{|Q^m_0|}\left(1-\frac{ |\partial_{A+D} F(S^i)|}{|F(S^i)|} \right) \leq d_i \leq \frac{|F(S^i)|}{|Q^m_0|}\left(1 + \frac{ |\partial_{A+D} F(S^i)|}{|F(S^i)|}\right).\end{equation}

Recall that $|F(S^1)|=rsM \log (M)$, $|F(S^2)|= r'sM\log (M)$ and 
$|Q^m_0|= M\log_{N_0}(M)\frac{|Q^0_0|}{N_0}=M\log_{N_0}(M)\frac{n}{N_0}$, hence 
$$ \frac{|F(S^1)|}{|Q^m_0|}= rs\frac{N_0}{n \log (N_0)} \quad \textrm{and} \quad 
 \frac{|F(S^2)|}{|Q^m_0|}= r's\frac{N_0}{n \log (N_0)}.$$

Now, since $F(S^1)$ and $F(S^2)$ are both F\o lner sets, we know that for any $\delta>0$ 
there exists $M_0$ such that if $M > M_0$, then for both $i=1$ and $i=2$ we have 
$$\frac{ |\partial_{A+D} F(S^i)|}{|F(S^i)|} < \delta.$$
Notice that the constant $M_0$ depends on $\delta$ and the quasi-isometry constants $K,C$. 
We fix $\delta > 0$ so that $|d_1 - d_2| > \frac{3 L^2 \delta N_0}{n \log (N_0)}.$  
For $i=1$, Equation (\ref{eqneqn}) becomes
$$rs\frac{N_0}{n \log (N_0)}\left(1-\delta \right) \leq d_1 \leq rs\frac{N_0}{n \log(N_0)}\left(1 +\delta\right),$$
and for $i=2$,
$$r's\frac{N_0}{n \log(N_0)}\left(1-\delta \right) \leq d_2 \leq r's\frac{N_0}{n \log(N_0)}\left(1 +\delta\right).$$ 
Thus, as $s \leq L$ and $r \leq L$,
\bea d_2 - d_1 &\leq&  r's\frac{N_0}{n \log(N_0)}\left(1 +\delta\right) -rs\frac{N_0}{n \log(N_0)}\left(1-\delta \right)\\
&\leq& (r'-r) s\frac{N_0}{n \log(N_0)}\  +\delta s\frac{N_0}{n \log(N_0)}(r'+r)\\
&\leq& (r'-r) s\frac{N_0}{n \log(N_0)}\  +\delta 2L^2\frac{N_0}{n \log(N_0)}.\\
\eea
%
Therefore, as $s \leq L$, for $\varepsilon_0 = \frac{\delta L}{2}$ we must have $|r-r'| > \varepsilon_0$.
\end{proof}

\subsection{No $K$-bilipschitz maps}
We now show by contradiction that for any $K$ there is no $K$-biLipschitz map $f: \D \to \Gamma$. 
Recall that each $K$-biLipschitz map $f$ is bounded distance $A$ from a companion map of the form 
$F(x,y,t)=(f_\ell(x), f_u(y), t)$, where $f_\ell,f_u$ are $L$-biLipschitz maps of $\R$. 

\begin{lemma}\label{stretchlemma} Let $f: \mathcal{D} \to \Gamma$ be a map such that for all $i=1, \ldots , M$, 
\begin{equation}\label{eq1}
\frac{|f_\ell((i-1)M) - f_\ell(iM)|}{M} \leq (1+\lambda)     \frac{|f_\ell(0) - f_\ell(M^2)|}{M^2}.
\end{equation}
Then for some $k$ we have both
$$\frac{|f_\ell((k-1)M) - f_\ell(kM)|}{M} \geq (1- \lambda)     \frac{|f_\ell(0) - f_\ell(M^2)|}{M^2},$$
$$\frac{|f_\ell(kM) - f_\ell((k+1)M)|}{M} \geq (1- \lambda)     \frac{|f_\ell(0) - f_\ell(M^2)|}{M^2}.$$

\end{lemma}
\begin{proof}
If at least half of the intervals get stretched by less than $(1-\lambda)$, we must 
have some other interval stretched by more than $(1+\lambda)$,  which is impossible 
by our hypothesis. Therefore, by the pigeonhole principle, there are at least two 
consecutive intervals that get stretched by more than $(1- \lambda)$.
\end{proof}

\begin{lemma}\label{getM02}  
There exist $\lambda_0,M_0$ 
such that if  $M=N_0^{2^m}\geq M_0$, $\lambda \leq \lambda_0$ and 
\begin{equation}\label{eqnone}
\frac{|f_\ell((i-1)M) - f_\ell(iM)|}{M} 
\leq 
(1+\lambda)     \frac{|f_\ell(0) - f_\ell(M^2)|}{M^2}
\end{equation}
holds for $i=1, \ldots, M$, then $f:\D \to \Gamma$ cannot be $K$-biLipschitz. 
\end{lemma}

\begin{proof}
The choice of $M$ defines F\o lner sets for $i=1, \ldots, M:$
$$S^i = [(i-1)M,iM) \times [0,1) \times [0, \log{M}).$$
Since $M=N_0^{2^m}$, all $S^i$ are subsets of $S_{M^2}$. 

Using Lemma \ref{stretchlemma}, we find $k$ such that  both inequalities below hold:
\begin{eqnarray}\label{thearray1}
\frac{|f_\ell((k-1)M) - f_\ell(kM)|}{M} &\geq& (1- \lambda)     \frac{|f_\ell(0) - f_\ell(M^2)|}{M^2},\\
\label{thearray2}
\frac{|f_\ell(kM) - f_\ell((k+1)M)|}{M} &\geq& (1- \lambda)     \frac{|f_\ell(0) - f_\ell(M^2)|}{M^2}.
\end{eqnarray}

Now either  $|\D \cap S^k|= |Q_1^m|=d_1|Q^m_0|$ and $|\D \cap S^{k+1}|= |Q_2^m|=d_2|Q^m_0|$, 
or vice versa. At the same time, {\color{black}up to post-composition with isometries,} we have that for some $r,r'$,  
$$F(S_k)= U_{rM}\times U_s \times [0, \log{M})$$ 
and 
$$F(S_{k+1})=U_{r'M}\times U_s \times [0, \log{M}).$$ 
By Equations (\ref{eqnone}), (\ref{thearray1}) and (\ref{thearray2}), 
\bea
|r-r'| &\leq& (1+ \lambda )\frac{|f_\ell(0) - f_\ell(M^2)|}{M^2} -(1-\lambda)\frac{|f_\ell(0) - f_\ell(M^2)|}{M^2} \leq 2\lambda L.
\eea
Let $M_0$ and $\epsilon_0$ be as specified by Lemma \ref{getM0} and let $\lambda_0 =\frac{\epsilon_0}{2L}$.
Then we have that if $M \geq M_0$ and $\lambda \leq \lambda_0$, then $f: \D \to \Gamma$ cannot be $K$-biLipschitz. 
\end{proof}

\begin{lemma}\label{last step} Suppose $f_\ell: \R \to \R$ is a map. If there exist $M_0$ and $\lambda_0$ such that for all $M \geq M_0$ and $\lambda \leq \lambda_0$ we have that for at least one $i \in \{ 1, \ldots, M \}$, 
$$ \frac{|f_\ell((i-1)M) - f_\ell(iM)|}{M} > (1+\lambda) \frac{|f_\ell(0) - f_\ell(M^2)|}{M^2},$$
then $f_\ell$ cannot be $L$-Lipschitz for any $L$. 
\end{lemma}
\begin{proof} Suppose $f_\ell$ is $L$-Lipschitz. Pick $t$ such that 
$$\frac{(1+\lambda)^t}{L} >  {L},$$
and recursively define $M_j=M_{j-1}^2$ for $j=1, \ldots, t$. By assumption, for some $i_t$ we must have 
$$\frac{ |f_\ell((i_t-1)M_{t-1}) - f_\ell(i_t M_{t-1})|}{M_{t-1}} > (1+\lambda) \frac{|f_\ell(0) - f_\ell(M_t)|}{M_t}.$$
Repeating this process, we get  
\bea
\frac{ |f_\ell((i_1 -1)M_{0}) - f_\ell(i_1 M_{0})|  }{M_0}
&>&  (1+\lambda) \frac{|f_\ell((i_{2}-1)M_1) - f_\ell(i_{2}M_1)|}{M_1}\\
   & \vdots &  \hspace{1in} \vdots\\
 &>& (1+\lambda)^{2}\frac{ |f_\ell((i_t-1)M_{2}) - f_\ell(i_tM_{t-1})| }{M_{t-1}} \\
&>& (1+\lambda)^t \frac{|f_\ell(0) - f_\ell(M_t)|}{M_t}. 
\eea
But then
$$ {L} \geq \frac{ |f_\ell((i_1-1)M_{0}) - f_\ell(i_1 M_{0})|  }{M_0}
>  (1+\lambda)^t \frac{|f_\ell(0) - f_\ell(M_t)|}{M_t}\geq \frac{(1+\lambda)^t}{L} > L,$$
which is a contradiction. 
\end{proof}

%
%
%

\begin{theorem}\label{sol:thm} Any lattice in $SOL$ contains a non-rectifiable Delone subset.\end{theorem}
\begin{proof} Let $\mathcal{\D} \subset \Gamma$ be the set constructed in Section \ref{tiling:sec} and suppose that $f:\mathcal{D} \to \Gamma$ is a $K$-biLipschitz map. Let $F(x,y,t)=(f_\ell(x), f_u(y), t)$ be the companion map to $f$.
Then $f_\ell$ must be $L$-biLipschitz for some $L$. 
Now by Lemma \ref{getM02}, we must have that there exist $M_0$ and $\lambda_0$ 
such that if $M\geq M_0$ and $\lambda \leq \lambda_0$, then for some $i\in\{1, \ldots, M\}$, 
$$ \frac{|f_\ell((i-1)M) - f_\ell(iM)|}{M} > (1+\lambda) \frac{|f_\ell(0) - f_\ell(M^2)|}{M^2}.$$
However, Lemma \ref{last step} would then show that in this case $f_\ell$ cannot be $L$-biLipschitz for any $L$, which is a contradiction.
\end{proof}

\section{Higher Rank analogues of SOL}\label{hr:sec}


In this section we generalize our arguments from lattices in $SOL$ to lattices in a class of abelian-by-abelian solvable Lie groups.  
Namely, let $G_\phi=\R^{n+1} \rtimes_\phi \R^{n}$, where $\phi: {{\R^{n}}} \to SL_{n+1}(\R)$ can be simultaneously diagonalized so that for each $\mathbf{t}\in\R^{n}$ the map $\phi(\mathbf{t}):\R^{n+1} \to \R^{n+1}$ is multiplication by the exponential of 
$$\bm \alpha_1(\mathbf{t}) & 0 &  \cdots & 0 \\ 0 & \alpha_2(\mathbf{t})  & 0  & \vdots \\ \vdots& 0 &  \vdots & 0 \\  0 & \cdots&   0 & \alpha_{n+1}(\mathbf{t}) \fm.$$
We call $\alpha_i: \R^{n} \to {{\R}}$ the  \emph{roots} associated to $\phi$. 
We  consider the case where $\alpha_i$ has the form $\alpha_i(\mathbf{t})= a_i t_i$ for $i \leq n$, with $a_i>0$, so that $\alpha_{n+1}(\mathbf{t})= -(a_1t_1+ \cdots + a_{n}t_{n})$. Our theorem also holds in the slightly more general case where some of the $a_i$ are negative but we omit this more general case primarily for the sake of ease of notation. 
In order to be able to construct tilings, however, we further restrict the $a_i$ so there exists a $t\in \R$ such that $e^{a_i t} \in 2\Z$ for all $i$. We call such an abelian-by-abelian Lie group \emph{even-scaling}.
 

In the case where, for each $i$, the \emph{rootspace} $V_{\alpha_i}$ associated to $\alpha_i$  (i.e.
 the subspace on which $\phi$ acts by multiplication by $e^{\alpha_i(\mathbf{t})}$ for all $\mathbf{t} \in \R^{n}$ ) 
is one dimensional, such an abelian-by-abelian Lie group will be called \emph{boundary one-dimensional}. 
 We will write $x_i$ for the coordinate representing $V_{\alpha_i}$. 
From Peng \cite{P1,P2}, we have the following theorem on the structure  of self quasi-isometries of such a $G_\phi$.

\begin{theorem}[Peng]\label{pengorigthm}Any self $(K,C)$ quasi-isometry of $ G_\phi=\R^{n+1} \rtimes_\phi \R^{n}$ is, up to permuting the rootspaces, at bounded distance $A$ from a  map of the form
$$(f_1 \times \cdots \times f_{n+1}) \times id$$
where $f_i$ is a $L$-biLipschitz map of $V_{\alpha_i}$.\end{theorem} 

Again, we have that for any Delone set $\D \subset G_\phi$ and any lattice $\Gamma \subset G_\phi$, every biLipschitz map 
$f: \D \to \Gamma$ induces a self quasi-isometry of $G_\phi$ that is bounded distance from a map of the form specified by Theorem \ref{pengorigthm}. As before, we call $F$ the \emph{companion} quasi-isometry to $f$.
Without loss of generality, we can assume that $F$  does not permute the root spaces.
 Additionally, we can assume that $f$ coarsely fixes the identity, and then $L=L(K,C)$ and $A=A(K,C)$.


\subsection{F\o lner sets.}
We define standard F\o lner sets following (with slight modifications) the definitions given in 2.2.4 of \cite{P1}, where these F\o lner sets are called \emph{boxes}.
We first chose $\Omega=[0,1]^{n} \subset \R^{n}$ and note that $\Omega$ satisfies the \emph{well rounded} condition of 2.2.4 of \cite{P1}. 
Note also that for $i\leq n$ we have  $\alpha_i(\Omega)=[0, a_i]$, whereas $\alpha_{n+1}(\Omega)=[-\sum_{i=1}^n a_i,0]$.
We define $a_{n+1}:=0$ to simplify notation below.

Let  $b_{\alpha_j}(t)=[0,t] \subset V_{\alpha_j}$. (This slightly differs from the definition of $b_{\alpha_j}(t)$ in \cite{P1}).
Define $B(\Omega) $, the \emph{box} associated to $\Omega$, as the union of left translates of $\Omega$ over all elements of 
$\prod_{j=1}^{n+1} b_{\alpha_j}(e^{\max{\alpha_j(\Omega)}})$, that is, 
$$B(\Omega) := \left(\prod_{j=1}^{n+1} b_{\alpha_j} (e^{\max{\alpha_j(\Omega)}})\right) \Omega.$$
{{With a similar definition for $r \Omega$}}, we get that 
$$B(r\Omega)= \left(\prod_{j=1}^{n+1} b_{\alpha_j} (e^{ra_j})\right)r\Omega= \left(\prod_{j=1}^{n+1} [0,e^{ra_j}]\right)  r\Omega = \left(\prod_{j=1}^{n} [0,e^{ra_j}]\right) \times [0,1]\times  r\Omega .$$

Now, since the volume element is given by 
$$e^{-\alpha_1(\mathbf{t})}dx_1 \wedge \cdots \wedge e^{-\alpha_{n+1}(\mathbf{t})}dx_{n+1} 
\wedge d\mathbf{t}=dx_1 \wedge \cdots \wedge dx_{n+1} \wedge d\mathbf{t},$$ 
the volume of a box $B(r\Omega)$ is given by
 $$|B(r\Omega)|=\left(\prod_{j=1}^n e^{ra_j}\right)r^n |\Omega|=\left(\prod_{j=1}^n e^{ra_j}\right)r^n.$$
 
\begin{lemma}[Lemma 2.2.7 in \cite{P1}]
\begin{equation}\label{eqn:folner} \left|\frac{\partial_\epsilon B(r\Omega)}{B(r\Omega)} \right|= O\left(\frac{\epsilon}{\diam(B(r\Omega))}\right).\end{equation}
\end{lemma}
Since a companion quasi-isometry stretches the various root spaces while keeping $\Omega$ fixed, we need to understand what happens to the volume and boundary of the box when we alter $B(r\Omega)$ slightly by keeping $\Omega$ fixed but changing  the size of the intervals in $\R^{n+1}$.

Given $u := (u_1, \ldots, u_{n+1})$, define the \emph{$u$-modified box} to be 
 $$B_{u}(r\Omega) :=\left(\prod_{j=1}^{n+1} [0, e^{ra_j + u_j}]\right) r\Omega.$$
Then 
$$|B_{u}(r\Omega)|=\left(\prod_{j=1}^{n+1} e^{ra_j+u_j}\right)r^n|\Omega|=e^{u_1+ \cdots + u_{n+1}}|B(r\Omega)|.$$ 
The following lemma shows that we also get behavior similar to Equation (\ref{eqn:folner}).

\begin{lemma}\label{lem:border}
Let $\|u\| := \sum |u_i|$. Then
\begin{equation} 
\label{eqn:ufolner} \frac{|\partial_\epsilon B_u(r\Omega)|}{|B_u(r\Omega)|}= O\left(\frac{\epsilon e^{\|u\|}}{\diam(B(r\Omega))}\right).\end{equation}
\end{lemma}

\begin{proof}
We follow the same proof as in \cite{P1} and bound the ratio of $\frac{ |\partial_\epsilon B_u(r\Omega)|}{|B_u(r\Omega)|}$ in terms of $\frac{ |\partial_\epsilon B(r\Omega)|}{|B(r\Omega)|}$. This will allow us to apply Lemma 2.2.7 of \cite{P1}. 

First note that the boundary $|\partial_\epsilon B_u(r\Omega)|$ can be decomposed as follows:\footnote{{Strictly speaking, there is an 
extra term arising from combination of points in $\partial_{\varepsilon} \left( \prod_j [0,e^{ra_j+ u_j}] \right)$ and $\partial_{\varepsilon} (r \Omega)$, 
but the corresponding volume is negligeable.}}
$$\left| \partial_{\varepsilon} \left( \prod_j [0,e^{ra_j+ u_j} ](r \Omega) \right) \right| 
= \underbrace{\left| \partial_{\varepsilon} \left( \prod_j [0,e^{ra_j+ u_j}] \right)(r \Omega)  \right|}_{(1)} 
+ \underbrace{ \left|  \left( \prod_j [0,e^{ra_j+ u_j}] \right)\partial_{\varepsilon} (r \Omega) \right|}_{(2)}.$$
When $u_i=0$ for all $i$, we are in the case of Lemma 2.2.7 of \cite{P1}, and we will refer to these terms as 
$(1)'$ and $(2)'$. Next we estimate each term separately:
\begin{enumerate}
\item[(2)]:  $\displaystyle \left|  \left( \prod_j [0,e^{ra_j+ u_j}] \partial_{\varepsilon} (r \Omega) \right) \right| 
=  \left( \prod_j e^{ra_j+ u_j}\right) r^{n-1}|\partial_{\varepsilon} \Omega|
= e^{u_1+ \cdots u_{n+1}}  \left|  \left( \prod_j [0,e^{ra_j}] \partial_{\varepsilon} (r \Omega) \right) \right|$
\item[(1)]:  
$\displaystyle \left| \partial \left( \prod_j [0,e^{ra_j+ u_j}] \right)(r \Omega)  \right|$
\bea
 &=& 
2 {{\varepsilon}} 
\sum_{j=1}^{n+1} \int_{\mathbf{t} \in r \Omega} \underbrace{\int_0^{e^{ra_1+u_1}}\hspace{-.2in}\cdots \int_0^{e^{ra_{n+1}+u_{n+1}} } e^{-\alpha_1(\mathbf{t})}dx_1 \ldots e^{-\alpha_{n+1}(\mathbf{t})}dx_{n+1}  }_{i\neq j}d{\mathbf{t}} \\
 &=& 
2 {{\varepsilon}} 
\sum_{j=1}^{n+1} \left( \prod_{i \neq j} e^{ra_i+u_i}\right)\int_{\mathbf{t} \in r \Omega} \underbrace{ e^{-\alpha_1(\mathbf{t})}dx_1 \ldots e^{-\alpha_{n+1}(\mathbf{t})}dx_{n+1} }_{i\neq j} d{\mathbf{t}} \\
 &=& 
2 {{\varepsilon}}
\sum_{j=1}^{n+1} \left( \prod_{i \neq j} e^{ra_i+u_i}\right)  \int_{\mathbf{t} \in r \Omega} e^{\alpha_j(\mathbf{t})} d{\mathbf{t}} \\
 &=& 
2 {{\varepsilon}} 
\sum_{j=1}^n \left( \prod_{i \neq j} e^{ra_i+u_i}\right)  {{r^{n-1}}}  \int_{0}^{r} e^{a_jt_j} dt_j + 
2 {{\varepsilon}} 
\left( \prod_{i=1}^n e^{ra_i+u_i}\right)  \int_{\mathbf{t} \in r\Omega} e^{-t_1a_1} e^{-t_2 a_2} \cdots e^{-t_n a_n} d\mathbf{t}\\
&=& 
2 {{\varepsilon}}  
\sum_{j=1}^n \left( \prod_{i \neq j} e^{ra_i+u_i}\right)  {{r^{n-1}}} a_j^{-1} (e^{ra_j} -1) +
2 {{\varepsilon}} \left( \prod_{i=1}^n e^{ra_i+u_i}\right) (-1)^na_1^{-1}\cdots a_n^{-1} (e^{-ra_1} -1) \cdots (e^{-ra_n}-1)\\
&=& 
2 {{\varepsilon}}  
\sum_{j=1}^n \left( \prod_{i \neq j} e^{ra_i+u_i}\right)  {{r^{n-1}}}  a_j^{-1} (e^{ra_j} -1) + 
2 {{\varepsilon}} \left( \prod_{i=1}^n -a_i^{-1}(e^{-ra_i} -1)e^{ra_i+u_i}\right ) \\
&=& 
2 {{\varepsilon}}  
e^{u_1+ \cdots + u_{n+1}}{{r^{n-1}}} \sum_{j=1}^n e^{-u_j}a_j^{-1} (e^{ra_j} -1)\left( \prod_{i \neq j} e^{ra_i}\right)  + 
2 {{\varepsilon}} e^{u_1 + \cdots + u_{n+1}}\left(e^{-u_{n+1}} \prod_{i=1}^n a_i^{-1}(e^{ra_i}-1)\right ).\\
\eea
\end{enumerate}
Thus, in our case we have 
\bea 
\frac{|\partial_\epsilon B_u(r\Omega)|}{|B_u(r\Omega)|}  & =& \frac{(1)}{|B_u(r\Omega)|} + \frac{(2)}{|B_u(r\Omega)|} 
= \frac{(1)}{e^{u_1+ \cdots + u_{n+1}}|B (r\Omega)|} + \frac{(2)}{e^{u_1+ \cdots + u_{n+1}}|B(r\Omega)|}.
\eea
In the second term, the expression $e^{u_1+ \cdots u_{n+1}}$ cancel, thus giving  
$$ \frac{(2)}{|B_u(r\Omega)|} =  \frac{(2)'}{|B(r\Omega)|}.$$
In the first term, if for all $i$ we have $u_i\geq 0$, then  
$$ \frac{(1)}{|B_u(r\Omega)|} \leq \frac{(1)'}{|B(r\Omega)|}.$$ 
In the opposite case, if for all $i$ we have $u_i<0$, then 
$$\frac{(1)}{|B_u(r\Omega)|} \leq e^{-(u_1 +\cdots +u_{n+1})} \frac{(1)'}{|B(r\Omega)|}.$$
In general,  
\bea 
\frac{|\partial_\epsilon B_u(r\Omega)|}{|B_u(r\Omega)|}  &  \leq & e^{|u_1| + \cdots + |u_{n+1}|}  \frac{|\partial_\epsilon B(r\Omega)|}{|B(r\Omega)|}   = O\left(\frac{\epsilon e^{\|u\|}}{\diam(B(r\Omega))}\right),
\eea
as announced.
\end{proof}

The following corollary states that for $n=2$ and $a_1=1, a_2=-1$, this covers the case of $SOL$. 

\begin{corollary}\label{sol:corollary} Let $G_\phi=\R^2 \rtimes \R$ be $SOL$. Then 
$$ \frac{|\partial_{\varepsilon} B_u(r\Omega)|}{|B_u(r\Omega)|} \leq2 \frac{(e^{-u_1}+e^{-u_2})}{r}+ \frac{1}{e^r}.$$
\end{corollary} 

\begin{proof} For $SOL$, a box $B(\Omega)$ is given by $\Omega = [0,1]$, 
so that $B(r \Omega)= [0, e^r] \times [0,1] \times [0,r]$, 
with volume $re^r$. Likewise, for $u := (u_1,u_2)$, we have
$$B_{u}(r\Omega)=[0, e^{r+u_1}] \times [0,e^{u_2}] \times [0,r],$$
with volume $|B_{u}(r\Omega)| = e^{u_1+u_2} r e^r$.  By investigating the calculations in our previous lemma, we see that 
$$|\partial B_u(r\Omega)|= 2e^{u_1+u_2}(e^{-u_1}(e^r-1)+e^{-u_2}(e^r-1)+ r)$$
So that 
\bea \frac{|\partial B_u(r\Omega)|}{|B_u(r\Omega)|} &=& \frac{2e^{u_1+u_2}(e^{-u_1}(e^r-1)+e^{-u_2}(e^r-1)+ r)}{e^{u_1+u_2} r e^r}\\
&=& \frac{2((e^{-u_1}+e^{-u_2})(e^r-1)+ r)}{ r e^r}\\
&\leq& 2\frac{(e^{-u_1}+e^{-u_2})}{r}+ \frac{1}{e^r}, 
 \eea
as announced.
\end{proof}

Notice that for $N=e^r$ (which fits with the set $S_N$ previously considered in $SOL$), this ratio estimate becomes 
$2\frac{(e^{-u_1}+e^{-u_2})}{\log{N}}+ \frac{1}{N} \leq 3\frac{(e^{-u_1}+e^{-u_2})}{\log{N}}$.

\subsection{Tiling} Our tiling is an extension of the tiling of $SOL$. As mentioned above, for our tiling to work we 
need to chose $a_i$ for which there exists $t\in \R$ such that $e^{a_i t} \in 2\Z$ for all $i$. If we define a standard 
F\o lner set by
$$S_t = \prod_{j=1}^{n+1}[0, e^{a_j t}] \times \prod_{i=1}^{n} [0,t],$$
then  
$S_{2t}$ can be tiled by $2^{n}e^{t(\sum_{i=1}^{\textcolor{black}n} a_i)}$ many translates of $S_t$. (We have changed notation slightly 
from the case of $SOL$ replacing $\log{N_0}$ by $t$ since we must account for the different weights $a_i$.) Now we define our 
basic tiles $Q_0,Q_1,Q_2,Q_3,Q_4$ as follows. We start as before by choosing $t=T$ large enough so that $S_{T} \cap \Gamma$  
contains at least $12$ points that are at least distance $C$ from the boundary of $S_{T}$. 
We pick $12$  of these points and call the result $Q_0$.  Then, for $j=1, 2, 3, 4$, we chose $Q_j\subset Q_0$ to have density $d_j$, 
where $d_j$ satisfies the conditions specified before in Section \ref{tiling:sec}. (For example, $d_1=1/3, d_2=1/2,d_3=1/4, d_4=7/12$.) 

Again, to define $Q^1_j \subset S_{2T}$ for $j=0, 1, 2, 3, 4$ from our basic tiles, we tile $S_{2T}$ by translates of the $Q_j$'s as follows.  First we tile $S_{2T}$ by translates of $S_T$.  Replacing each translate of $S_T$ with a translate of $Q_0$ (by the same element) defines the tile $Q_0^1$. For the other tiles we proceed as follows. Any translate that has as its first coordinate the interval $[ke^{a_1T}, (k+1)e^{a_1T}]$  with $k$ even is colored red.  If $k$ is odd this translate is colored blue. The rest of the translates are colored green.  

{\bf Claim.} Half of the tiles are in $S_{2T}$ are colored green, a quarter red and a quarter blue. 

\begin{proof} We have $S_{2T}=  (\prod_{j=1}^{n+1}[0, e^{a_j 2T}]) \times [0,2T]^n \subset \R^{n+1} \rtimes \R^n$ and the translates of $S_T$ that cover $S_{2T}$ correspond to sub-cubes of $[0,2T]^n$, where each interval in the sub-cube is either $[0,T]$ or $[T,2T]$.
The translates of $S_T$ with first coordinate $[ke^{a_1T}, (k+1)e^{a_1T}]$ correspond to those translates with sub-cube with first coordinate $[0,T]$, which is exactly half of all translates. Half of these have $k$ odd and half have $k$ even. 
\end{proof}

As before, to form $Q^1_1$ we replace the red translates with translates of $Q_1$, the blue translates with translates of $Q_2$, and the green translates with translates of $Q_3$. To form ${Q^1_2}$, we do the same except we replace the green translates with translates of 
$Q_4$. The sets $Q^1_3, Q^1_4$ are formed by replacing all of the colors with translates of $Q_3$ and $Q_4$ respectively. Then, as before, $|Q^1_j|=d_j|Q^1_0|$ for all $j=1, 2, 3, 4$.  Iterating this procedure allows us to define $Q^m_j$ for any $m$. Finally, moving points a distance $\leq C$ apart if necessary, this defines for us a Delone set $\D\subset \Gamma$. (As before, whichever parts of $G_\phi$ we have not colored we can just replace with their intersection with $\Gamma$). 

Next we need an analogue of Lemma \ref{getM0}. In the statement of this lemma, we define for any $t=2^mT$ the set 
$\bar{S}_t:=Q^m_0$. Then $|\bar{S}_t|= C'|S_t|$  for some constant $C'$ depending on $|Q_0|$.


\begin{lemma}\label{1.2new} There exist $t_0 > 0$ and $\epsilon_0 > 0$, both depending on $K,C$, satisfying the following: 
Suppose that  $t > t_0$ and that $S^1,S^2$ are translates of the standard F\o lner set $S_t$  where $t={2^m}T$ for some 
$m$, and $|\mathcal{D}\cap S^i|=d_i|\bar{S}_t|$. 
Suppose also that $f: \mathcal{D} \to \Gamma$ is $K$-biLipschitz, and that the companion quasi-isometry $F$ satisfies, 
for $i=1,2$ and up to postcomposition with isometries, 
$$F(S^i)= B_{u^i}(t\Omega).$$
Then $\|u^1- u^2\| > \epsilon_0$. 
\end{lemma}

\begin{proof} 
Recall that  $|S_t|= (\prod_{j=1}^{n+1} e^{a_j t})t^{n}$.
By  Lemma \ref{first:lemma}, we have that 
$$ |F(S^i)| - |\partial_{A+D} F(S^i)|  \leq |f(\mathcal{D} \cap S^i)| \leq |F(S^i)| + |\partial_{A+D} F(S^i)|.$$
Hence, 
$$ |B_{u^i}(t\Omega)| - |\partial_{A+D} B_{u^i}(t\Omega)| \leq C' d_i |S_t| \leq |B_{u^i}(t\Omega)| + |\partial_{A+D} B_{u^i}(t\Omega)|,$$
and so
$$ \frac{|B_{u^i}(t\Omega)|}{|S_t|} -  \frac{|B_{u^i}(t\Omega)|}{|S_t|}\frac{ |\partial_{A+D} B_{u^i}(t\Omega)|}{|B_{u^i}(t\Omega)|} 
\leq 
C' d_i 
\leq 
\frac{|B_{u^i}(t\Omega)|}{|S_t|} +\frac{|B_{u^i}(t\Omega)|}{|S_t|}\frac{ |\partial_{A+D} B_{u^i}(t\Omega)|}{|B_{u^i}(t\Omega)|} $$
that is, 
$$ \frac{|B_{u^i}(t\Omega)|}{|S_t|} \left(1 - \frac{ |\partial_{A+D} B_{u^i}(t\Omega)|}{|B_{u^i}(t\Omega)|} \right)
 \leq 
C' d_i 
\leq 
\frac{|B_{u^i}(t\Omega)|}{|S_t|} \left(1 + \frac{ |\partial_{A+D} B_{u^i}(t\Omega)|}{|B_{u^i}(t\Omega)|} \right).$$
By Equation (\ref{eqn:ufolner}), for $t>t_0$ large enough we have that for both $i=1,2$,
$$\frac{ |\partial_{A+D} B_{u^i}(t\Omega)| }{|B_{u^i}(t\Omega)|} < \delta.$$
Thus,
$$ \frac{|B_{u^i}(t\Omega)|}{|S_t|} \left(1 - \delta \right) \leq d_i 
\leq \frac{|B_{u^i}(t\Omega)|}{|S_t|} \left(1 +\delta \right),$$
and therefore
\bea
C' | d_1 - d_2 | 
& \leq  &  \frac{|B_{u^1}(t\Omega)|}{|S_t|}  (1+ \delta) -  \frac{|B_{u^2}(t\Omega)|}{|S_t|} (1 -\delta)\\
&\leq & \left| \frac{|B_{u^1}(t\Omega)| - |B_{u^2}(t\Omega)|}{|S_t|} \right| + \delta\left| \frac{|B_{u^1}(t\Omega)| + |B_{u^2}(t\Omega)|}{|S_t|} \right|.
\eea 

Recall also that 
$$|B_{u^i}(t\Omega)|=\left(\prod_{j=1}^{n+1} e^{ta_j+u^i_j}\right)t^{n}=e^{\sum_{j=1}^{n+1} u^i_j }\left(\prod_{j=1}^{n+1} e^{ta_j}\right)t^{n}=e^{\sum_{j=1}^{n+1} u^i_j }|S_t|,$$ 
hence 
\begin{eqnarray*}
\left| \frac{|B_{u^1}(t\Omega)| \pm |B_{u^2}(t\Omega)|}{|S_t|} \right|&\leq& \left|{e^{\sum_{j=1}^{n+1} u^1_j }} \pm {e^{\sum_{j=1}^{n+1} u^2_j }} \right|  \\
& \leq & {e^{\sum_{j=1}^{n+1} u^i_j }}\left|e^{\sum_{j=1}^{n+1} (u^1_j - u^2_j )} \pm1\right|\\
& \leq & {e^{\sum_{j=1}^{n+1} u^i_j }}\left|e^{\|u^1 - u^2\| } \pm 1\right|\\
& \leq & L^{n+1}\left|e^{\epsilon } \pm 1\right|,
\end{eqnarray*}
where $\varepsilon := \| u^1 - u^2 \|$. The last step follows from the fact that all boundary maps $f_j$ 
are $L$-biLipschitz, and so $e^{u_j^i} \leq L$  for all $j$. 

Summarizing, we have 
$$ C' |d_1-d_2| \leq L^{n+1}|e^{\epsilon} -1|+\delta L^{n+1} |e^{\epsilon} +1|.$$
In this inequality, $d_1,d_2, C'$ and $L$ are given, and $\delta$ can be made arbitrarily small by choosing $t > t_0$ for a 
large enough $t_0$. Fixing such a $\delta$ very small, we conclude that $\epsilon > \epsilon_0$ for some appropriately 
chosen $t_0,\epsilon_0$.
\end{proof}

\subsection{No $K$-bilipschitz maps}
The following is a modification of Lemma \ref{getM02}. 

\begin{lemma}\label{gett0}  
There exist $\lambda_0>0$ and $M_0>0$ 
such that if 
\begin{equation}\label{eqnone-new}
\frac{|f_1((i-1)M) - f_1(iM)|}{M} 
\leq 
(1+\lambda)     \frac{|f_1(0) - f_1(M^2)|}{M^2}
\end{equation}
holds for $i=1, \ldots, M$, where $\lambda < \lambda_0$ and $M := e^{a_1t}$ (with $t=T2^m$)  
satisfies $M \geq M_0$, then $f:\D \to \Gamma$ cannot be $K$-biLipschitz. 
\end{lemma}

\begin{proof}
The choice of $M$ defines F\o lner sets for $i=1, \ldots, M$:
$$S^i=[(i-1)M, iM] \times \left( \prod_{j=2}^{n+1} [0,e^{ta_j}] \right) \times [0,t]^n.$$
These are all subsets of $S_{T2^{m+1}}$. As in Lemma \ref{stretchlemma}, there exists 
$k$ such that  both inequalities below hold:
\begin{eqnarray}
\label{thearray1-new}
|f_1((k-1)M) - f_1(kM)| &\geq& (1- \lambda)     \frac{|f_1(0) - f_1(M^2)|}{M},\\
\label{thearray2-new}
|f_1(kM) - f_1((k+1)M)| &\geq& (1- \lambda)     \frac{|f_1(0) - f_1(M^2)|}{M}.
\end{eqnarray}

Now either  $|\D \cap S^k|= |Q_1^m|=d_1|\bar{S}_{t}|$ and $|\D \cap S^{k+1}|= |Q_2^m|=d_2|\bar{S}_{t}|$, 
or vice versa. We also have that for some $u^1,u^2$ {{satisfying $u^1_j = u^2_j$ for $j=2,\ldots,n+1$}}, 
the set  $F(S^k)$ is a translate of $B_{u^1}(t\Omega)$, and $F(S^{k+1})$ is a translate of $B_{u^2}(t\Omega)$. 
By equations (\ref{eqnone-new}), (\ref{thearray1-new}) and (\ref{thearray2-new}),  
\bea
|e^{u^1_1}-e^{u^2-1}| &\leq& (1+ \lambda )\frac{|f_1(0) - f_1(M^2)|}{M^2} -(1-\lambda)\frac{|f_1(0) - f_1(M^2)|}{M^2} \leq 2\lambda L.
\eea
Let $t_0$ and $\epsilon_0$ be as in Lemma \ref{1.2new}, set $M_0=e^{t_0a_1}$, and assume $e^{ta_1} = M \geq M_0$. 
{{Then we have
$$e^{\varepsilon_0}-1 \leq | e^{u^1_1} - e^{u^2_1} | \leq 2 \lambda L.$$
However, if  $\lambda \leq \lambda_0$ for a very small $\lambda_0$, this 
yields a contradiction. }} 
\end{proof}

\begin{theorem}\label{hr:thm} If $\Gamma \subset G_\phi= \R^{n+1} \rtimes_\phi \R^n$ is a lattice (where $G_\phi$ is a boundary 
one-dimensional even scaling abelian-by-abelian solvable Lie group as defined above), then there exist non-rectifiable Delone sets in $\Gamma$.
 \end{theorem}

\begin{proof} 
Let $\mathcal{D}$ as indicated above and suppose that $f:\mathcal{D} \to \Gamma$ is a $K$-biLipschitz map. Let $F=(f_1, \ldots, f_{n+1}, id)$ be the companion quasi-isometry to $f$. Then $f_1$ must be $L$-biLipschitz for some $L$. Now by Lemma \ref{gett0} we must have that there exist $M_0$ and $\lambda_0$ such that if $M\geq M_0$ and $\lambda \leq \lambda_0$, then for some $i\in\{1, \ldots, M\}$,  
$$ \frac{|f_1((i-1)M) - f_1(iM)|}{M} > (1+\lambda) \frac{|f_1(0) - f_1(M^2)|}{M^2}.$$
But then proceeding as in Lemma \ref{last step}, we see that $f_1$ cannot be $L$-biLipschitz for any $L$, which is a contradiction. 
\end{proof}

\section{Solvable Baumslag-Solitar groups}\label{bs:sec}

In this section, we show how the methods of the first two sections can also be used to produce non-rectifiable 
Delone sets in the solvable Baumslag-Solitar groups 
$$BS(1,m)= \left< t, a \mid tat^{-1}=a^m\right>.$$
The Baumslag-Solitar groups are not cocompact lattices in any real Lie group. 
However, they can be viewed as cocompact lattices in the locally compact isometry group of a fibered product $X_m$ given by the following diagram:
\[
\begin{tikzcd}
X_m \arrow{dr}{\bar{h}} \arrow{r}{\rho_1} \arrow[swap]{d}{\rho_2} & T_{m+1} \arrow{d}{h} \\
\mathbb{H}_m^2 \arrow{r}{b}&\R
\end{tikzcd}
\]
Here, $\mathbb{H}_m^2$ is a hyperbolic plane  with curvature $-\frac{1}{\ln(m)}$ having 
$b:\mathbb{H}_m^2 \to \R$ as a Busemann function, and $T_{m+1}$ is an $(m+1)$-valent simplicial tree, with
$h:T_{m+1} \to \R$ a height function defined by fixing an orientation on edges such that each vertex has one 
incoming edge and $m$ outgoing edges.
  We  also refer to $b:\mathbb{H}^2_m\to \R$ and the induced map $\bar{h}:X_m \to \R$ as height functions on $\mathbb{H}^2_m$ and  $X_m$ respectively. 
Topologically, the space $X_m$ can be identified with $T_{m+1} \times \R$, but metrically it is isometric to a family of hyperbolic planes glued together along horoball complements at integer heights.  For a more detailed description of $X_m$ see \cite{FM1, FM2}. 
By fixing a base point $x_0$ with $\bar{h}(x_0)=0$ and considering the orbit of $x_0$ under $BS(1,m)$, we can embed the Cayley graph of $BS(1,m)$ into $X_m$. 

{\bf Remark.} The group $SOL$ has a similar description to the above with $\mathbb{H}^2_m$ and $T_{m+1}$ replaced by two hyperbolic planes $\mathbb{H}^2$ and $h$ replaced by $-b$, the negative of a Busemann function $b: \mathbb{H}^2 \to \R$. 


\begin{lemma}  
There exists a metric on 
$$\{ (x,y, t) \in \R \times \Q_m \times \R\}$$
and two  $(K,C)$ quasi-isometries with $K=C=1$ that are coarse inverses of each other 
$$\pi:  \R \times \Q_m \times \R \to X_m, \quad  \bar{\pi}: X_m\to  \R \times \Q_m \times \R,$$  
such that $\bar{\pi}$ is injective on $BS(1,m)\subset X_m$.
\end{lemma}

The above lemma will be proved as part of the proof of the following theorem. 

\begin{theorem}[Farb-Mosher \cite{FM2}]
For any $(K,C)$ quasi-isometry $\phi: X_m \to X_m$, the map 
$$\pi \circ \phi \circ \bar{\pi}: \R \times \Q_m \times \R \to \R \times \Q_m \times \R$$ 
is at bounded distance $A=A(K,C)$ from a map $F$ of the form 
$$F(x,y,t) =(f_\ell(x), f_u(y), t),$$
where the $f_\ell, f_u$ are $L$-biLipschitz maps of $\R$ and $\Q_m$ respectively and where the 
constant $L$ depends only on $K$ and $C$ provided that some base point is $C$-coarsely fixed by $f$. 
\end{theorem}


\begin{proof} All of this can be found in \cite{FM1,FM2} albeit using slightly different language and so we reconcile these differences here. In \cite{FM1},  Farb and Mosher first define a \emph{lower boundary} $\partial_\ell X_m$ and \emph{upper boundary}  $\partial_u X_m$ as equivalence classes of \emph{vertical geodesics} in $X_m$. Vertical geodesics are those geodesics $\gamma(t)$ that project isometrically to $\R$ under $\bar{h}$ and are parametrized so that $\bar{h}(\gamma(t))=t$. Two vertical geodesics are equivalent in the lower boundary if they stay a bounded distance apart in $X_m$ as $t \to -\infty$, and they are equivalent in the upper boundary if they stay bounded distance apart as $t\to \infty$.  One can identify $\partial_\ell X_m \simeq \partial \mathbb{H}_m^2\setminus \{\infty \} \simeq \R$ and $\partial_u X_m \simeq \partial T_{m+1}\setminus \{\infty \} \simeq \Q_m$. To specify a vertical geodesic in $X_m$, is it enough to chose $x \in \partial_\ell X_m \simeq \R$ and $y \in \partial_u X_m \simeq \Q_m$. To specify a point on that geodesic, one simply specifies its height $t\in \R$. 
This defines a natural projection map
$$\pi: \R \times \Q_m \times \R \to X_m,$$
where $(x,y,t)$ is mapped to the the point at height $t$ in $X_m$ on the vertical geodesic with lower boundary endpoint $x \in \partial_\ell X_m$  and upper boundary endpoint $y \in \partial_u X_m$. 
Note that this map is not injective. Indeed any point $(x,y',t)$ with $y'\in B_{\Q_m}(y,r)$ where $-\log_m r=t$ also defines the same point in $X_m$. We can define a nice coarse inverse map $\bar{\pi}: X_m \to \R\times\Q_m \times \R$  by sending $v \in X_m$ to an arbitrary point $(x,y,t)$ with $\pi(x,y,t)=v$. The only choice one has in defining $\bar{\pi}$ is in the $y$ coordinate.  The following picture illustrates a possible choice for the $\bar{\pi}$ map restricted to the tree coordinate.
%
%
%
%
%
%
%
%
%

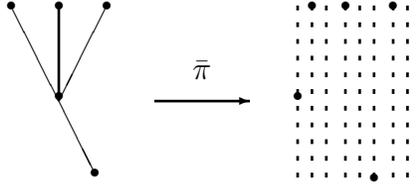
\begin{figure}[h]
\center
\setlength{\unitlength}{.25in} 
\begin{picture}(9,4.6)(1,.5) 
\linethickness{1pt} 

\put(1,4.6){\line(1,-2){1.78}}
\put(2,4.6){\line(0,-1){2}}
\put(3,4.6){\line(-1,-2){1}}

\put(1,4.6){\circle*{.15}}
\put(2,4.6){\circle*{.15}}
\put(3,4.6){\circle*{.15}}

\put(2,2.7){\circle*{.15}}
\put(2.75,1.1){\circle*{.15}}

\put(4,2.6){\vector(2,0){2}}  
\put(5,3.25){\makebox(0,0){$\bar{\pi}$}} 

\multiput(7,1)(0,.35){11}{\line(0,1){0.1}} 
\multiput(7.3,1)(0,.35){11}{\line(0,1){0.1}} 
\multiput(7.6,1)(0,.35){11}{\line(0,1){0.1}} 
\multiput(8,1)(0,.35){11}{\line(0,1){0.1}} 
\multiput(8.3,1)(0,.35){11}{\line(0,1){0.1}} 
\multiput(8.6,1)(0,.35){11}{\line(0,1){0.1}} 
\multiput(9,1)(0,.35){11}{\line(0,1){0.1}} 
\multiput(9.3,1)(0,.35){11}{\line(0,1){0.1}} 
\multiput(9.6,1)(0,.35){11}{\line(0,1){0.1}} 

\put(7.3,4.6){\circle*{.15}}
\put(8,4.6){\circle*{.15}}
\put(9,4.6){\circle*{.15}}
\put(7,2.7){\circle*{.15}}
\put(8.6,1){\circle*{.15}}

\end{picture}
\caption{A possible choice for $\bar{\pi}$ shown only in projection onto $T_{3+1}$ and $\Q_3 \times \R$.}
\end{figure}


We endow $\bar{X}_m$  with a metric that is coarsely equivalent to the induced metric from $X_m$: 
$$d_{\bar{X}_m} \big( (x_1,y_1,t_1), (x_2,y_2,t_2) \big) 
:= m^{-\frac{t_1+t_2}{2}} |x_1 - x_2| 
+ m^{\frac{t_1+t_2}{2}} d_{\mathbb{Q}_m}(y_1, y_2) + |t_1 - t_2|.$$


In \cite{FM1,FM2} it is shown that every quasi-isometry  $\phi:X_m\to X_m$ preserves height level sets up to bounded distance and  maps vertical geodesics to within bounded distance of vertical geodesics. They use this information to define induced boundary maps $f_\ell, f_u$ and conclude that these maps are biLipschitz. 
\end{proof}





We then set $\bar{X}_m:=\R \times \Q_m \times \R$ and identify $BS(1,m)$ with its image $\bar{\pi}(BS(1,m))$.  
If we choose our coordinates so that  the identity in $BS(1,m)$ has coordinates $(0,0,0)$ then any $v \in BS(1,m)$ has coordinates $(x,y,t)$ with $x \in \Z[\frac{1}{m}]\subset \R$, $y \in \Z[\frac{1}{m}] \subset \Q_m$ and $t \in \Z$. 

Note also that there is an action of $BS(1,m)$ on $\bar{X}_m$ given by its action on vertical geodesics in $X_m$. Similarly to 
$\Gamma \subset SOL$, one can show (see Lemma \ref{folBS:lemma} below) that F\o lner sets in $BS(1,m)$ can be given by 
intersecting with $BS(1,m)$ translates of the set
$$S_{N} := [0,N) \times B(0,1) \times [0, \log_m N),$$
where $B(0,1)$ is a ball of radius one in $\Q_m$.
When $N$ is a power of $m$ we have  
$$|BS(1,m)\cap S_N|=N \log_m{N},$$
but otherwise $|BS(1,m)\cap S_N|=\lfloor N \rfloor \lfloor \log_m N \rfloor$.
We will write $|S_N|$ for $|BS(1,m) \cap  S_N|$ in either case.


The figure below shows a F\o lner set  in $BS(1,2) \subset X_2$, where a portion of a $\mathbb{H}_2^2 \subset X_2$ component is on the left, a portion of the projected view onto $T_{2+1}$  is on the right, and the projected view onto $\mathbb{H}^2_2$ (showing all points) is in the middle. 

\begin{figure}[htbp]
\begin{center}
\includegraphics[scale=.5]{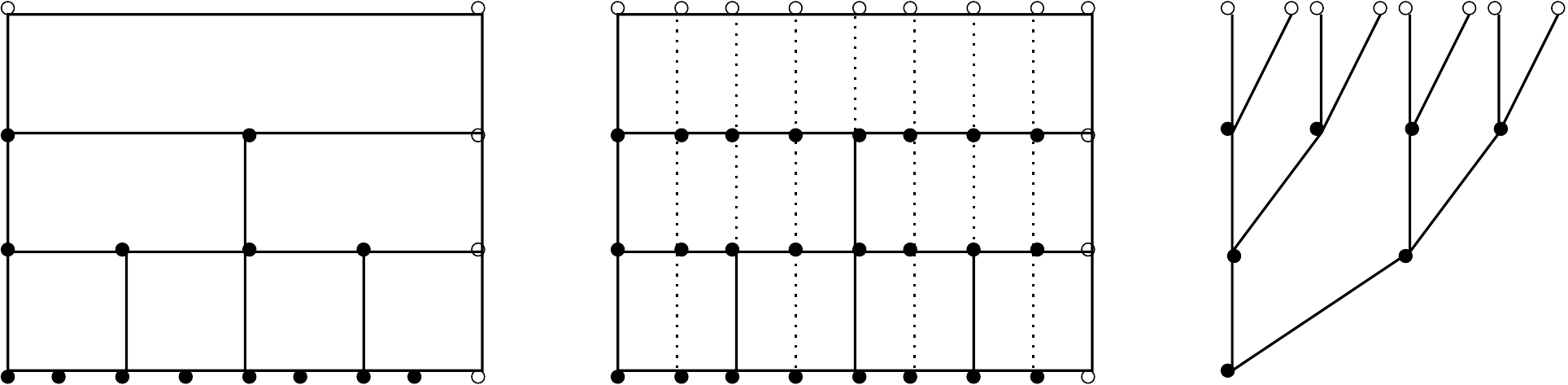}
\caption{A F\o lner set in $BS(1,2)$ containing $3 \cdot 2^3$ elements.}
\label{default}
\end{center}
\end{figure}

The following lemma shows that $S_N\cap BS(1,m)$ is indeed a F\o lner sequence. The slightly more general statement will be useful in later proofs. 

\begin{lemma}\label{folBS:lemma} Suppose $U_{rN}$  is an interval of length $rN$. Then any set of the form 
$$S_N'=U_{rN} \times B(0,1) \times [0, \log_m{N})$$ 
for a fixed $r$ is a F\o lner set.
\end{lemma}

\begin{proof} To compute the boundary of $S'_N$ it is useful to consider $\pi(S'_N)$. Doing this, and arguing as in Lemma \ref{lem:border}, 
it becomes clear that the $1$-boundary of $S'_N$ can be decomposed into (1) and (2) below:
\begin{enumerate}
\item $|\partial_1 ( U_{rN} {{\times B(0,1))}}|$: this is twice the number of vertices in a rooted tree of height 
$\log_mN-1$ with branching constant $m$, that is, $2\frac{N-1}{m-1}$; 
\item  $|\partial_1 [0, \log_mN)|$: this is twice the number $\lfloor rN \rfloor$.  
\end{enumerate}
{{Since $BS(1,m)$ is generated by two elements, an obvious argument gives}} 
that the $R$-boundary of $S'_N$ is bounded by 
$$ |\partial_R S'_N|\leq e^{4R} \left( 2\lfloor rN\rfloor + 2 \frac{N-1}{m-1} \right).$$
Since the volume of $S'_N$ is given by $\lfloor  rN \rfloor  \log_mN$, this yields the ratio
$$ \frac{|\partial_R S'_N |}{|S'_N|} \leq \frac{ e^R (2\lfloor rN\rfloor + 2 \frac{N-1}{m-1})}{\lfloor  rN \rfloor  \log_mN} \sim \frac{1}{\log_m N}
$$
which goes to $0$ as $N \to \infty$. 
\end{proof}

Next we need to estimate the size of $|F(S_{N}) \cap BS(1,m)|$.

\begin{lemma}\label{volstr:lemma} If $F(S_N)= U_{rN} \times U_s \times [0, \log_m N)$, 
then for large enough $N$,
$$ \left(\frac{1}{L^2}- \epsilon \right)|S_N| \leq |F(S_N)| \leq L^4 |S_N|,$$ 
where $\epsilon \to 0$ as $N \to \infty$. 
\end{lemma}

\begin{proof} 
Recall that $F = (f_{\ell}, f_u, id)$, where $f_{\ell}, f_u$ are $L$-biLipschitz. In particular, $1/L \leq r \leq L$.
As noted in the proof of Proposition 4.3 in  \cite{D}, the image of any ball in $\Q_m$ under an $L$-biLipschitz 
map is the disjoint union of finitely many balls of bounded size, where both the number of balls and their size 
are bounded by constants that depend on $L$ and the size of the original ball.
Specifically, $U_s=f_u(B(0,1))=\bigsqcup_{i \in I} A_i$, where $A_i$ are balls of 
size at least $|B(0,1)|/L=1/L$. Indeed, if $z \notin B(0,1)$, then $d_{\Q_m}(z,B(0,1)) > 1$ 
and therefore, for some $A_i$,
\begin{equation}\label{cloneeqn} d_{\Q_m}(f_u(z),A_i) =d_{\Q_m}(f_u(z),f_u(B(0,1)))> \frac{1}{L}.
\end{equation}
This implies that $|A_i| >\frac{1}{L}$, otherwise there would be some $y$ with
$d_{\Q_m}(y,A_i) \leq \frac{1}{L}$, and letting $z$ be such that $y := f_u(z)$ 
(remind that $f_u$ is surjective), this would contradict inequality (\ref{cloneeqn}). 

Without loss of generality, the $A_i$ can be chosen to be all of the same size.
Then since $|I||A_i|=|f_u(B(0,1))|$, we must have $\frac{1}{L}\leq|I||A_i|  \leq L $, and combining this with the inequality 
$|I|\frac{1}{L}<|I| |A_i|$, we get that $|I|< L^2$. We can further assume that $|A_i|=m^{-j}
\leq 1$ where $j$ depends only on $L$. (Since $|A_i|\leq L$ we can always achieve this by writing each $A_i$ as a union of at most 
$L$ balls of size at most $1$.) This increases $|I|$ by at most a factor of $L$, {{hence $|I| \leq L^3$.}} There is also 
a lower bound $\frac{m^j}{L}\leq |I|$, since $\frac{1}{L} \leq |I| |A_i|$.


Next we estimate the size of $U_{rN} \times A_i \times [0, \log_m N)$. 
First we estimate the size of the subset  
$$U_{rN} \times A_i \times [ -\log_m{|A_i|}, \log_m N)=U_{rN} \times A_i \times [ j , \log_m N).$$ 
For this subset, the choice of $\bar{\pi}$ does not change the number of points in 
$$BS(1,m) \cap (U_{rN} \times A_i \times [ j , \log_m N),$$ 
which is exactly $\lfloor rN  m^{-j} \rfloor  (\log_m N  -j)$. Therefore,
$$(rNm^{-j}-1)( \log_m N -j) \leq |U_{rN} \times A_i \times [ j , \log_m N)| \leq rN  m^{-j}   (\log_m N -j).$$
This gives us the lower bound 
\bea |F(S_N)| & \geq&  \sum_{i\in I} (rNm^{-j}-1) (\log_m N -j) \\
& \geq & \frac{m^j}{L}  (rNm^{-j}-1) (\log_m N -j)\\
& \geq &  \left(\frac{r}{L}- \frac{m^j}{LN} - \frac{jr}{L\log_mN}\right)|S_N|\\ 
& \geq & \left(\frac{1}{L^2}- \epsilon\right)|S_N|,
\eea
where the last inequality holds for large enough $N$ for a prescribed $\varepsilon > 0$.

%
For the upper bound, we have to consider the extra points that may appear in the interval $[0,j)$. Here, for each $A_i$ 
there will be at most $j (rN)$ points, so that 
\bea
|F(S_N)|& \leq & \sum_{i\in I}\left( rN  m^{-j}   (\log_m N -j) + jrN\right)\\
& \leq &  L^3 rN  \log_m N  -L^3 rN j  + L^3 jrN\\
& \leq & L^3 r |S_N| \\
&\leq& L^4 |S_N|
\eea 
as desired.
\end{proof}

\begin{lemma}\label{folner:lemma} The image of the standard F\o lner set $S_N$ under $F$ is also a F\o lner set.
\end{lemma} 

\begin{proof} 
We use some of the work from the proof of the previous lemma. As before, let 
$$F(S_N)= U_{rN} \times \bigsqcup_{i \in I} A_i \times [0, \log_m N),$$
with $|A_i|= m^{-j}$ and $\frac{m^j}{L} \leq |I|\leq L^3$.
Fix $i$ and let $S'= U_{rN} \times A_i \times [0, \log_m N)$. 
We will give an upper bound for $|\partial_R S'|$ and multiply this by $|I|$ to get an upper bound for $\partial_R F(S_N)$; 
then, we will combine it with the lower bound for $|F(S_N)|$ from Lemma \ref{volstr:lemma} to get the desired result. To 
that end, we decompose $S'$ as follows
$$S' =  U_{rN} \times A_i \times [j, \log_mN) \bigcup  U_{rN} \times A_i \times [0,j). $$
Note that $U_{rN} \times A_i \times [j, \log_mN)$ is a translate of the set
$$U_{rNm^{-j}} \times B(0,1)\times [0, \log_mN -j)=U_{rNm^{-j}} \times B(0,1)\times [0, \log_m (Nm^{-j}) ).$$
This allows us to appeal to Lemma \ref{folBS:lemma} to get that 
$$|\partial_R \left(U_{rN} \times A_i \times [j, \log_mN)\right)| \leq e^{4R} \left(2\lfloor rm^{-j}N\rfloor + 2 \frac{Nm^{-j}-1}{m-1}\right).$$
Next, for $\partial_R (U_{rN} \times A_i \times [0,j) )$ we have the estimate 
$$|\partial_R (U_{rN} \times A_i \times [0,j) )| \leq e^{4R} rN  j,$$
since in particular $|U_{rN} \times A_i \times [0,j) | \leq rNj$. 
Putting everything together we get that 
$$\frac{|\partial_R F(S_N) |}{|F(S_N)|} 
\leq 
\frac{e^{4R} \left(2\lfloor rm^{-j}N\rfloor + 2 \frac{Nm^{-j}-1}{m-1}\right)+e^{4R} rN  j }{ \big( \frac{1}{L^2} - \varepsilon \big) |S_N|}$$
which goes to $0$ as $N \to \infty$.
\end{proof}


\subsection{Tiling.} Constructing $\mathcal{D}$ can be done exactly as for $SOL$. 
For $N$ a power of $m$, we consider the set $S_N$,  
so that $S_N\cap BS(1,m)$ contains exactly $N\log_m{N}$ elements. 
As before, we can tile the set $S_{N^2}$ by $2 N$ translates of $S_N$ in two layers of $N$ copies each:
$$T_{k,1}:=g_{k,1}S_N=[(k-1)N,kN) \times B(0,1) \times [0, \log_m{N}),$$
$$T_{k,2}:=g_{k,2}S_N=[0,N^2) \times B \big( z_k, \frac{1}{N} \big) \times [ \log_m{N}, 2 \log_m{N}).$$
Here, both $g_{k,1}$ and $g_{k,2}$ lie in $BS(1,m)$, and the points $z_k= \sum a_i m^i \in \mathbb{Q}_m$ satisfy $a_i=0$ for $i<0$ and $i\geq\log_m{N}$, 
and $a_0,a_2, \ldots, a_{\log_m{N}-1}$ lie in the set $\{0, \ldots, m-1 \}$. (Note that there are $m^{\log_m N}=N$ possibilities for $(a_0,a_2, \ldots, a_{\log_m{N}-1})$, hence $N$ points $z_k$.)

So now we fix $N_0$ a power of $m$ so that $|S_{N_0} \cap BS(1,m)|\geq12$ and let $Q_0=S_{N_0}\cap BS(1,m)$. (For $BS(1,2)$, the value of $N_0$ 
can be $8$, for example.)  Our four basic tiles $Q_1,Q_2,Q_3,Q_4$ can be defined by choosing
$d_1=1/3, d_2=1/2,d_3=1/4, d_4=7/12$ of the points from $Q_0$. 
If $m$ is even, then our tilling works exactly as before: to define $Q^1_1$  (resp. $Q^1_2$) we use $N_0/2$ translates of $Q_1$ (for each $T_{k,1}$ with $k$ even), $N_0/2$ translates of $Q_2$ (for each $k$ odd) and $N_0$ translates of $Q_3$ (resp. $Q_4$) for each $T_{k,2}$.  However, when $m$ is odd, $N_0$ will be odd as well, and so if we have $\lceil N_0/2 \rceil$ translates of $Q_1$ and $\lfloor N_0/2 \rfloor$ translates of $Q_2$, then if we replace all $T_{k,2}$ with $Q_3$ (resp. $Q_4$), the volume $|Q^1_1|$ (resp. $|Q^1_2|$) will not be exactly $d_1 |Q^1_0|$ (resp. $d_2|Q^1_0|$). To solve this problem, by replacing one of the $Q_3$ (resp. $Q_4$) tiles with a tile of slightly different density, we can again ensure that 
$Q^1_j= d_j|Q^1_0|$ for $j=1,2$.  Specifically, we replace a $Q_3$ (resp. $Q_4$) tile with a tile of density $d_3+\frac{1}{2}(d_2-d_1)$  (resp. $d_4 +\frac{1}{2}(d_2-d_1)$). One can check that if the $d_j$ are given as above, then this new tile has density $1/3$ (resp. $2/3$). In both cases, $Q^1_3$ and $Q^1_4$ are defined as before, by replacing all the tiles with $Q_3$ (resp. $Q_4$).

{\bf Remark.} When $m$ is odd, there is another modification that needs to be made in the proof. Namely, Lemma \ref{stretchlemma} which we appeal to later in the proof does not hold with an odd number of intervals. However, it does hold if we modify the lower bounds to $(1-2\lambda)$. This will only change the constants slightly in the subsequent lemmas and propositions. 

The following lemma, which is an analogue of Lemma \ref{getM0}, follows almost exactly as before.

\begin{lemma}\label{getM03}
There exist $M_0 > 0$ and $\epsilon_0 > 0$, both depending on $K,C$, satisfying the following: 
Suppose $M={m^j}> M_0$ and suppose $S^1,S^2$ are standard F\o lner sets in $BS(1,m)$ that are translates of $S_M$ 
with $|\mathcal{D}\cap S_1|=d_1 |S_{M}|$ and $ |\mathcal{D}\cap S_2|=d_2|S_{M}|$.  
Suppose also that $f: \mathcal{D} \to BS(1,m)$ is $K$-biLipschitz, and that its companion quasi-isometry $F$ satisfies
$$F(S^1)= U_{rM} \times U_s \times [0,\log_m{M})$$
$$F(S^2)=U_{r'M} \times U_s \times [0,\log_m{M})$$
where $U_s= \bigsqcup_{i \in I} A_i$. 
Then $|r'-r| > \epsilon_0$. 
\end{lemma}

\begin{proof}
As before, since 
$$ F(S^i) \setminus   \partial_{A}F(S^i) \subseteq  f(S^i\cap \mathcal{D}) \subseteq F(S^i) \cup \partial_{A}F(S^i),$$
we have that 
$$ |F(S^i)| - |\partial_{A}F(S^i)|  \leq d_i |S_{M}| \leq |F(S^i)| + |\partial_{A} F(S^i)|,$$
and so
$$ \frac{|F(S^i)|}{|S_{M}|} \left(1 - \frac{ |\partial_{A}F(S^i)|}{|F(S^i)|} \right) \leq d_i 
\leq \frac{|F(S^i)|}{|S_{M}|} \left(1 + \frac{ |\partial_{A} F(S^i)|}{|F(S^i)|} \right).$$
By Lemma \ref{folner:lemma}, we have that for any $\delta>0$ we can chose $M$ large enough so that  
\begin{equation}\label{eq:nec}
 \frac{|\partial_R F(S_{M})|}{|F(S_{M})|} < \delta
\end{equation}
and by  Lemma  \ref{volstr:lemma} we know that for $i=1,2,$ 
$$ |F(S^i)|<L^4|S_M|.$$
Thus, 
\bea 
d_1-d_2 &\leq&  \frac{|F(S^1)|}{|S_M|} (1+\delta) -  \frac{|F(S^2)|}{|S_M|} (1-\delta)\\
& \leq & \frac{|F(S^1)|-|F(S^2)|}{|S_M|}  + \delta  \frac{|F(S^1)|+|F(S^2)|}{|S_M|} \\
& \leq & \frac{rsM \log_m{M} - r'sM\log_m{M}}{M\log_m{M}}  + 2\delta  L^4. \\
& \leq & (r - r')L  + 2\delta  L^4. 
\eea
Now pick $\delta < |d_1-d_2|/3L^4$ and $M_0$ such that for $M>M_0$ Equation (\ref{eq:nec}) holds. 
Then for $\epsilon_0 = \delta L^3$ we must have $|r-r'|>\epsilon_0$, otherwise the above inequality would be violated.
\end{proof}

\subsection{No $K$-bilipschitz maps}

We also need a slight modification of Lemma \ref{getM02}. 

\begin{lemma}\label{getsomething}  There exist $\lambda_0>0$ and $M_0>0$ 
such that if  $\lambda < \lambda_0$, $m^j=M\geq M_0$  and
\begin{equation}\label{eqnonenew}
\frac{|f_\ell((i-1)M) - f_\ell(iM)|}{M} 
\leq 
(1+\lambda)     \frac{|f_\ell(0) - f_\ell(M^2)|}{M^2}
\end{equation}
holds for $i=1, \ldots, M$, then $f:\D \to BS(1,m)$ cannot be $K$-biLipschitz. 
\end{lemma}

\begin{proof}
Again, the choice of $M$ defines F\o lner sets for $i=1, \ldots, M$:
$$S^i=[(i-1)M, iM) \times B(0,1) \times [0,\log_m{M}).$$
These are all subsets of $S_{M^2}$. 
By Lemma \ref{stretchlemma}, we can find $k$ such that  both inequalities below hold:
\begin{eqnarray}\label{thearray1new}
|f_1((k-1)M) - f_1(kM)| &\geq& (1- \lambda)     \frac{|f_1(0) - f_1(M^2)|}{M},\\
\label{thearray2new}
|f_1(kM) - f_1((k+1)M)| &\geq& (1- \lambda)     \frac{|f_1(0) - f_1(M^2)|}{M}.
\end{eqnarray}

Now either  $|\D \cap S^k|= |Q_1^j|=d_1|S_{m^{j_0}}|$ and $|\D \cap S^{k+1}|= |Q_2^j|=d_2|S_{m^{j_0}}|$, 
or vice versa.  
By equations (\ref{eqnonenew}), (\ref{thearray1new}) and (\ref{thearray2new}), 
\bea
|r-r'| &\leq& (1+ \lambda )\frac{|f_1(0) - f_1(M^2)|}{M^2} -(1-\lambda)\frac{|f_1(0) - f_1(M^2)|}{M^2} \leq 2\lambda L.
\eea
Let $M_0$ and $\epsilon_0$ be as in Lemma \ref{getM03} and $\lambda_0=\frac{\epsilon_0}{2L}$. Then we have that if $M \geq M_0$ and $\lambda \leq \lambda_0$ then $f: \D \to \Gamma$ cannot be $K$-biLipschitz. 
\end{proof}

\begin{theorem}\label{bs:thm} For any $m$, there exist non-rectifiable Delone sets in $BS(1,m)$.
\end{theorem}
\begin{proof} 
Let $\mathcal{D}$ as indicated above and suppose that $f:\mathcal{D} \to BS(1,m)$ is a $K$-biLipschitz map. Let $F=(f_\ell,  f_u, Id)$ be the companion quasi-isometry to $f$.
Then $f_\ell$ must be $L$-biLipschitz for some $L$. 
Now by Lemma \ref{getsomething}, there exist $M_0$ and $\lambda_0$ such that if $M\geq M_0$ and $\lambda \leq \lambda_0$, 
then for some $i\in\{1, \ldots, M\}$,  
$$ \frac{|f_1((i-1)M) - f_1(iM)|}{M} > (1+\lambda) \frac{|f_1(0) - f_1(M^2)|}{M^2}.$$
But then Lemma \ref{last step} shows that $f_\ell$ cannot be $L$-biLipschitz for any $L$. 
\end{proof}

{\bf Remark.} All of the groups $\Gamma$ we were able to treat in this paper had the following main feature in common: We were able to identify their quasi-isometry as products of bilipschitz maps on various \emph{boundary} metric spaces with at least one of those spaces being $\R$. Using this information we were able to show that under a quasi-isometry F\o lner sets were mapped to F\o lner sets and then to construct a tiling of the space in a way that made it impossible for a bilipschitz map $f: \mathcal{D} \to \Gamma$ to have a boundary map that was biLipschitz in $\R$.  It would be interesting to find examples of non-rectifiable Delone sets in abelian-by-abelian solvable Lie groups whose boundaries were not all one dimensional. 


\addcontentsline{toc}{subsection}{References}

\vspace{0.1cm}

Tullia Dymarz\\
University of Wisconsin-Madison, 480 Lincoln Drive, Madison, WI 53706\\
E-mail address: dymarz@math.wisc.edu

Andr\'es Navas\\
Dpto. de Matem\'atica y C.C., Univ. de Santiago, Alameda 3363, Santiago, Chile\\ 
E-mail address: andres.navas@usach.cl
\end{document}